\documentclass[11pt]{amsart}
\usepackage{amsmath}
\usepackage{amssymb}
\usepackage{amsfonts}
\usepackage{amsthm}
\usepackage{verbatim}
\usepackage{amscd}
\usepackage{cite}
\usepackage{leftidx}
\usepackage{enumerate}
\usepackage{txfonts}
\usepackage{manfnt}
\usepackage{amscd}
\usepackage[mathscr]{eucal}
\usepackage{hyperref}
\usepackage{datetime2}
\usepackage{graphics}
\usepackage{mathpazo}
\newtheorem{thm}{Theorem} 

\newtheorem*{thm*}{Theorem}
\newtheorem*{prop*}{Proposition}
\newtheorem{cor}[thm]{Corollary}
\newtheorem*{cor*}{Corollary}

\newtheorem*{lem*}{Lemma}

\newtheorem*{claim*}{Claim}
\newtheorem{prop}[thm]{Proposition}

\theoremstyle{remark}

\newtheorem*{rem*}{Remark}
\newtheorem{crit-rem}[thm]{Critical remark}
\newtheorem{remarks}[thm]{Remarks}

\newtheorem{example}[thm]{Example}
\newtheorem*{example*}{Example}
  
\newtheorem*{defn*}{Definition}
\newtheorem*{con*}{Conjecture}

\def\refp #1.{(\ref{#1})}
\newcommand\carets [1]{\langle #1 \rangle}

\def\sbr #1.{^{[#1]}}
\def\sfl #1.{^{\lfloor #1\rfloor}}

\def\?{{\bf{??}}}

\def\H{\mathcal H}

\def\a{{\frak a}}

\def\P{\mathbb P}

\def\Z{\mathbb Z}

\def\L{\mathcal L}

\def\O{\mathcal O}

\def\Sym{\textrm{Sym}}

\def\rk{\text{rk}}

\def\g{\mathfrak g}

\def\1/2{\frac{1}{2}}

\def\2{{[2]}}

\def\nl{\newline}

\def\<{\langle}
\def\>{\rangle}

\def\2{{[2]}}

\def\scl #1.{^{\lceil#1\rceil}}
\def\spr #1.{^{(#1)}}
\def\sbc #1.{^{\{#1\}}}

\def\subpr#1.{_{(#1)}}

\def\beq{\begin{equation*}}
\def\eeq{\end{equation*}}

\def\g3{{\Gamma\spr 3.}}

\newcommand{\eqspl}[2]{
\begin{equation}\label{#1}
\begin{split}
#2\end{split}\end{equation}}

\newcommand{\exseq}[3]{
0\to #1\to #2\to #3\to 0
}

\newcommand{\beginalphaenum}{
\begin{enumerate}\renewcommand{\labelenumi}{ }
\item \begin{enumerate}
}

\def\eex{\end{rm}\end{example}}

\begin{document} 
	\title{Polarized interpolation and normal postulation for curves   on Fano hypersurfaces}
	\author 
	{Ziv Ran}
%
%
	\thanks{arxiv.org }
	\date {\DTMnow}
%
%
	\address {\nl UC Math Dept. \nl
	Skye Surge Facility, Aberdeen-Inverness Road
	\nl
	Riverside CA 92521 US\nl 
	ziv.ran @  ucr.edu\nl
	\url{https://profiles.ucr.edu/app/home/profile/zivran}
	}
	
	 \subjclass[2010]{14n25, 14j45, 14m22}
	\begin{abstract}
A general  hypersurface $X$ of degree $\leq n$ in projective space 
contains curves $C$ of any genus $g\geq 0$ and sufficiently large degree
depnding on $g$ whose normal and conormal bundles have good postulation  or natural cohomology in 
the sense that  each twist has either
$H^0=0$ or $H^1=0$. This implies a polarized version of the interpolation property for $C$ on $X$.

		\end{abstract}
	\maketitle
\section*{Introduction}\subsection{Set-up}	 The normal bundle
	$N=N_{C/X}$ is a fundamental 'linear' attribute associated to an embedding or immersion
	$C\to X$ of a curve in a variety, controlling the geometry and motions both local and global
	of $C$ in $X$. For $X=\P^n$ (and less so, a projective hypersurface)
	there have been many studies of normal bundles from various viewpoints including
	balancedness, interpolation and (semi) stability, see e.g. \cite{alyang}, \cite{semistable}.\par
	Our purpose here is to study $N$ and its dual conormal bundle $\check N$ from a 
	cohomological viewpoint closely related to interpolation. A bundle $E$ on a polarized curve
	$(C, \O(1))$ is said to have \emph{good postulation} (or sometimes \emph{natural cohomology})
	(with respect to $\O(1)$)
	if for each twist $E(i)$ one has either $H^0(E(i))=0$ (when $i\leq \chi(E)/\rk(E)$)
	 or $H^1(E(i))=0$ (when $i\geq\chi(E)/\rk(E)$). A curve $C$ on
	a polarized variety $(X, \O(1))$ is said to have \emph{normal maximal rank NMR} (resp.
	\emph{conormal maximal rank CMR}) if its normal bundle $N$ (resp. conormal bundle
	$\check N$) has good postulation with respect to the induced polarization.
	Geometrically, the NMR property is closely related to a
	'polarized' or  'constrained' analogue of interpolation 
	(see \S\ref{interpolation-sec}). \par
	In practice good postulation for $E$ is proven via the equivalent
	assertion  that for some $i_0$ one has
	\[H^0(E(i_0))=0, H^1(E(i_0+1))=0.\]
Note that in case $\chi(E(i_0))=0$, i.e. $\mu(E)=i_0(g-1)$, good postulation is
equivalent to $H^1(E(i_0))=0$. In this case $E(i_0)$ is said to be an Ulrich bundle
(see \cite{beauville-ulrich} \cite{lopez-ulrich}\cite{lopez-etal-ulrich}).
  Thus the good postulation condition may be considered a natural
generalization of (twisted) Ulrich beyond the case $\chi(E(i_0))=0$. 
\subsection{Results} The following statement summarizes our main results 
(Theorems \ref{g=1 thm}, \ref{g thm}, \ref{d=n thm}, \ref{subcanonical}):
\begin{thm*}
On a general hypersurface $X$
of degree $d\leq n$ in $\P^n$ there exist curves of any given genus $g$
and any sufficiently high degree $e$ depending on $g$ that have the NMR or CMR property.
\end{thm*}
This of course includes the case $d=1$, i.e. Projective space itself.
For example\par - in $\P^n$, for $g\geq 2$, NMR holds as soon as $e\geq n(13g-1)$ 
and CMR holds if $e\geq (5n+2)g$ (Theorem \ref{g thm}).\par
- If $X$ has degree $n$ in $\P^n$, NMR  and CMR hold if
$e\geq 12n^2g$ (Theorem \ref{d=n thm}).
\subsection{Methods} Though logically independent, the NMR property
seems analogous to stability of the normal bundle, with the role of slope replaced
by $\chi$: NMR involves subbundles $\O(i)\to N,
i\leq\chi(N)$, while
stability involves arbitrary subbundles $F\to N$ of given slope $\mu(F)\leq\mu(N)$. 
Thus some of the constructions used
to show stability in \cite{semistable} can be used here with modifications.\par
Thus we will make systematic use of the fang method, which  involves degenerating a hypersurface
$X\subset\P^n$ of degree $d$ (or $\P^n$ itself) to a reducible variety
\[X_0=X_1\cup X_2\subset P_0=P_1\cup_Q P_2\]
(or $\P^n$ to $\P_0$) where $P_1, P_2$ are 'complementary' blowups of $\P^n$
in $\P^m$ resp. $\P^{n-m-1}$ and $Q=\P^m\times\P^{n-m-1}$ is the common
exceptional divisor; and $X_0\subset P_0$ is a suitable hypersurface, e.g.
$X_1$ is typically a suitable blowup of $\P^{n-1}$ or a projective bundle over 
a lower projective space. Then in $X_0$ we study 
suitable reducible curves $C_0=C_1\cup C_2, C_i\subset X_i$.
In the case of $\P^n$ (\S \ref{genus 1}, \S \ref{genus g}) we use an induction on the genus. In the hypersurface
case (\S \ref{d=n},  \S \ref{d<n})  we typically take for $C_2$ a disjoint union of lines with trivial normal bundle
and for $C_1$ a birational model in $\P^{n-1}$.\par

I am grateful to Angelo Lopez for helpful comments.
\section{Normal postulation and polarized interpolation}\label{interpolation-sec}
Let $C_0$ be a lci  curve of degree $e$ on a smooth polarized variety $(X, \O(1))$.
To fix ideas, we assume the normal bundle $N=N_{C_0/X}$ has $H^0(N)\neq 0, H^1(N)=0$.
Let $\H$ be the Hilbert scheme parametrizing deformations of $C_0$ in $X$, 
with tangent space $T_{C_0}\H=H^0(N)$, and let
\[U_d\subset\H\times |\O(d)|\] be the open
subset consisting of pairs $(C, Y)$ that are transverse.
Let \[V_{ed}\subset\Sym^{ed}X\times|\O(d)|\] be the incidence variety consisting of pairs $(A, Y)$
where $A\subset Y$. We have an $|\O(d)|$-morphism
\eqspl{intersection}{\begin{matrix}
U_d&&\stackrel{J_d}{\to}&& V_{ed}\\
&\searrow&&\swarrow&\\
&&|\O(d)|&&
\end{matrix}\\
(C, Y)\mapsto (C\cap Y, Y))
} The induced tangent map $T_{J_d}$ on vertical tangent spaces over $|\O(d)|$ may
be identified as the restriction map
\[H^0(N)\to H^0(N|_{C\cap Y})\simeq H^0(T_Y|_{C\cap Y}).\]
and its kernel may be identified as $H^0(N(-d))$. If $T_{J_d}$ has maximal rank for all $d$,
$C$ or $C_0$ is said to satisfy \emph{polarized interpolation} (with respect  to $\O(1))$.
This means the the family of $de$-tuples $C\cap Y$ is as large as possible among tuples
lying on a degree-$d$ hypersurface and in particular is a general such $de$-tuple when the
numerology (i.e. $h^0(N_{C/X})$) permits. Thus we have
\begin{prop}
Notations as above, assume moreover $H^0(N_{C/X})\neq 0, H^1(N_{C/X})=0$. Then
 if $C_0$ has the NMR property, $C_0$ satisfied
polarized interpolation.
\end{prop}
\begin{remarks}
(i) For a bundle $E$ on a curve $C$ of genus 0, good postulation with respect to a line bundle
of degree $e$ means that $\mu_{\max}(E)-\mu_{\min}(E)\leq e$ where $\mu_{\max}, \mu_{\min}$
denote the largest degre\begin{flushleft}
e
\end{flushleft} of a sub (resp. smallest degree of a quotient) line bundle.\par
(ii) On a curve of genus 1 or more generally a subcanonical curve $C$ ($K_C=\O(k), k\in\Z$),
good postulation for $E$ and $\check E$ are equivalent.\par
(iii) For general $g$ the geometric interpretation of of the CMR property is less clear. However, if
$C$ has the CMR and $Y$ is a degree-$d$ hypersurface transverse to $C$ and containing
a canonical divisor $\a$, then either $Y.C\setminus\a$ consists of $de-2g+2$ general points
on $Y$ (if $H^1(\check N(d))=0$), or $C$ has no nontrivial deformations through $C.Y\setminus\a$
(if $H^0(\check N(d))=0$). This follows from 
Serre duality.
\end{remarks}
%
\section{case of genus 1 in $\P^n$}\label{genus 1}
The purpose of this section is to prove that general elliptic curves in $\P^n$ have NMR and CMR:
\begin{thm}\label{g=1 thm}
Let $C$ be a general elliptic curve of degree $e\geq 12n$ in $\P^n, n\geq 3$, with normal bundle $N$.
Then \par (A) if $e\geq 12n$ we have 
(i) $H^1(N(-1))=0$ and $N(-1)$ is generically generated.\par (ii) $H^0(N(-2))=0$.\par
(iii) If $N'\subset N$ is a down modification  at $\leq 2$ points,
then $H^1(N'(-1))=0$ and $N'$ is generically generated.\par
 (B) If $e\geq 4n+2$ we have\par 
(i) $H^0(\check N(1))=0$ .\par (ii) $H^1(\check N(2))=0$ and $\check N(2)$ is generically generated.\par
(iii) If $N'\subset \check N$ is a down modification at $\leq 2$ points,
then $H^1(N'(2))=0$ and $N'$ is generically generated.
\end{thm}
\begin{proof}(A) 
We begin by proving (i) in case $n=2m+1$ is odd (note that if $n=3$ then (i) actually follows from (ii)
and is anyhow redundant). Note that because $\chi(N(-1))=2e$, 
 our assertion is equivalent to $h^0(N(-1))\leq 2e$. As for (ii), when $e=2n-2$ or $e\geq 3n-3$,
 it follows from semistability of $N$ proven in \cite{semistable}, \S2.\par
We use a construction similar to that in \cite{semistable}, \S 5. Consider a fang of type $(n.m)$, i.e.
\eqspl{p0}{P_0=P_1\cup_QP_2}
where $P_1=P_2=B_{\P^m}\P^n, Q=\P^m\times\P^m$. This is a limiting form of $\P^n$.
On $P_0$ we consier the line bundle
\[\L_1=\O_{P_1}(1)\cup(\O_{P_2}(1)(-Q))\]
where $\O_{P_i}(1)$ denotes the pullback of $\O(1)$ from $\P^n$.
This has the property that $\L_1|_Q=\O(1,0)$ and it is a limit of $\O(1)$ on $\P^n$.\par
Now on $P_0$ we consider the following connected nodal lci genus-1 curve 
\eqspl{c0}{C_0=D_1\cup_{p_1,p_2} D_2\subset P_1\cup P_2,}
\[D_i=R_{i1}\cup_{q_{i1}}C_i\cup_{q_{i2}}R_{i2}, i=1,2,\]
where $D_i$ is the birational transform of a connected rational chain  $D'_i=R'_{i1}\cup C'_i\cup R'_{i2}$
consisting of rational curves of respective degrees  $e_i^b, e^d_i, d^b_i\geq n$ 
with $R_{1j}.Q=R_{2j}.Q=p_j, j=1,2$.
Then $C_0$ smooths out to a curve of genus 1 and degree $e=e^d_1+2e^b_1+e^d_2+2e^b_2-2$
in $\P^n$.
\par
Set $N_0=N_{C_0/P_0}(-\L_1)$. 
To begin with, clearly $N_0|_{C_i}=N_{D_i'/\P^n}|_{C'_i}(-1)$ is a rank-1 up modification 
of $N_{C_i/\P^n}(-1)$ at $q_{i1}, q_{i2}$, while $N_0|_{R_{1j}}$ (resp. $N_0|_{R_{2j}}$)
is a corank-$m$ down modification 
of a rank-1 up modification of $N_{R'_{1j}/\P^n}(-1)$ (resp. of $N_{R'_{2j}/\P^n}(-1)$, twisted by $+Q$). 
Therefore by general choices we may assume
$N_0|_{C_i}, N_0|_{R_{ij}}$ are balanced. Now assume $m$ is odd.
Then I claim we may further assume $N_0|_{R_{ij}}$ is perfect. Indeed the rank of the latter bundle is
$2m$ while its degree is
$2e_1^b-m-1$ for $ i=1$, $2e^b_2-m-1+2m=2e^b_2+m-1$ for $ i=2$,
 so it suffices to have $2e_1^b, 2e_2^b\equiv m+1\mod 2m$.
We take $e^b_1=e^b_2=(5m+1)/2$. Because $e^b_1, e^b_2\geq n$, the $R'_{ij}$
are nondegenerate so their normal bundles are balanced. This leads after the modifications to 
\[N_0|_{R_{1j}}=2m\O(2), N_0|_{R_{2j}}=2m\O(3), j=1,2.\]
Therefore if we write
\[2e_i^d=2ms_i+r_i, r_i\in [0,2m)\]
the we get
\eqspl{splitting}{N_0|_{D_1}&=N_{D_1/P_1}(-1)=r_1\O(2, s_1+1, 2)\oplus (2m-r_1)\O(2, s_1, 2),\\  
N_0|_{D_2}&=N_{D_2/P_2}(-1+Q)=r_2\O(3, s_2+1, 3)\oplus (2m-r_2)\O(3, s_1, 3).}
Clearly $H^0(N_{D_i/P_i})\to N_0|_{p_1, p_2}$ is surjective. This easily implies that
\eqspl{}{h^0(N_0)&=h^0(N_0(-\L_1)|_{D_1})+h^0(N_0(-\L_1)|_{D_2})-2m\\
&=(2(e_1^d+2e_1^b)-2)+(2(e_2^d+2e_2^b)-2+4m)-4m=2e
}
and that $N_0$ is generically generated by sections on $C_1$ and $C_2$.
Finally, a similar argument applies to a down modification $N'_0$,  of $N_0$ at  general points of $c_1\in C_1$,
$c_2\in C_2$, showing that $H^1(N'_0)=0$ and $N'_0$ is generically generated on $C_1$ and $C_2$..\par
This completes the proof in case $m$ is odd.
\par
If $m$ is even we can use a fang $P_0=P_1\cup_QP_2$ with $P_1=B_{\P^{m-1}}\P^n, P_2=B_{\P^{m+1}}\P^n$ and  
double locus $Q=\P^{m-1}\times\P^{m+1}$
and construct perfect bridges $R_{1j}, R_{2j}$ of respective degree $e^b_{1}\equiv m/2 \mod 2m$
on $X_1=B_{\P^{m+1}}\P^n$
and $e^b_{2}\equiv m/2+1\mod 2m$ on $X_2=B_{\P^{m-1}}\P^n$.
A similar construction also works for the case $n=2m$ even.\par
(B) Again we first assume $n=2m+1$. We consider a singular model $C_0$ as in \eqref{c0}
which only differs in the choice of bridges $R_{ij}$. Here we rake $R'_{ij}$ a general rational
curve of degree $m+1$, i.e. a rational normal curve in its span $\carets{R'_{ij}}$. We have
\[N_{R'_{ij}/\carets{R'_{ij}}}=m\O(m+3), N_{R'_{ij}/\P^n}=m\O(m+3)\oplus m\O(m+1).\]
Consequently
\[N_{R_{ij}/P_i}=m\O(m+2)\oplus m\O(m+1),\]
\[N_{C_0/P_0}|_{R_{ij}}=(m+1)\O(m+2)\oplus\O(m-1)\O(m+1).\]
As the limit of $\O(2)$ we take
\[\L_2=\O_{P_1}(2)(-Q)\cup\O_{P_2}(-Q)\]
so that $\L_2|_Q=\O(1,1).$ Thus we have
\[\check N_{C_0/P_0}(\L_2)|_{D_i}=\check N_{D'_i/\P^n}(2)=r_i\O(s_i)\oplus(2m-r_i)\O(s_i)\]
where $s_i=[(2m-2)e^d_i]$. From this it easily follow that 
$H^1(\check N_{C_i/P_i}(\L_2)(-p_1-p_2)=0$. Therefore
\[H^1(\check N_{C_0/P_0}(\L_2)=0.\]
This proves (i), and (iii) is proved similarly, while the proof of (ii) is similar and simpler,
using the same limits $\L$ of $\O(1)$ as above. This completes the case $n$ odd.
Then the case $n=2m$ even is handled similarly as above using a fang of type
$B_{\P^m}\P^n\cup B_{P^{m-1}}\P^n$.

/***********,
*********************/
\end{proof}
\section {Higher genus in projective space}\label{genus g}
\begin{thm}\label{g thm}
Let $C$ be a general  curve of genus $g\geq 2$ and degree $e$  in $\P^n, n\geq 3$, with normal bundle $N$.
and conormal bundle $\check N$ Then.\par (A) If $e\geq 13ng-n$, we have
\par If  (i) $H^1(N(-1))=0$ and $N(-1)$ is generically generated.\par (ii) $H^0(N(-2))=0$.\par
(iii) $N'\subset N$ is a down modification locally of corank $\leq n$ at $\leq 2$ points,
then $H^1(N'(-1))=0$ and $N'(-1)$ is generically generated.
\par (B) If $e\geq (5n+2)g, n\geq 5$ we have\par
(i) $H^1(\check N(2)=0)$ and $\check N(2)$ is generically generated.\par
(ii) $H^0(\check N(1))=0$.\par
(iii) If $N'\subset \check N$ is a down modification at $\leq 2$ points, then $H^1(N'(2))=0$
and $N'(2)$ is generically generated. 
\end{thm}
\begin{proof}(A)
The proof is by induction on $g$ using the case $g=1$ proved above. Again we begin with the case $n=2m+1$
odd. Consider a fang $P_0=P_1\cup_QP_2$ as in \eqref{p0} and a curve $C_0=C_1\cup C_2\subset P_1\cup P_2$,
such that: $C_1=C_{11}\cup C_{12}$ where $C_{11}, C_{12}\subset P_1$ are birational transforms
of general curves of genus 1 (resp. $g-1$) and degree $\geq 12n$ (resp. $\geq 13n(g-1)-n$)
meeting $\P^m$ in $p_1$ resp. $p_2$, and $C_2\subset P_2$ is the birational transform of a 
rational normal curve of degree 
$e_2= 2m+1$ meeting $\P^m$ in $p_1, p_2$. The $C_0$ smooths out to a curve of genus $g$ and degree
$\geq 13ng-n-2$ in $\P^n$.\par
 Pick general points $c_1\in C_{11}, c_2\in C_{12}$
and let $N_0'$ be a general down modification of $N_0=N_{C_0/P_0}$ at $c_1, c_2$.  By induction, we have
$H^1(N'_0|_{C_1}(-\L_1))=0$. By \cite{elliptic}, Lemma 31, , we know $N_0|_{C_2}$ is balanced which easily implies 
\[N'_0|_{C_2}=2m\O(2m+1),\]
hence $H^1(N'_0|_{C_2}(-p_1-p_2)(-\L_1))=0$.
This implies $H^1(N'_0)=0$. The rest follows easily. The case $n$ even is handled as in Theorem
\ref{g=1 thm}.\par
(B) We can use a similar inductive construction as above and Part (B) of Theorem \ref{g=1 thm}. We have
in case $n=2m+1$,
$\check N_0|_{C_2}(\L_2)=2m\O(2m-4)$
so we can conclude as above.
\end{proof}

A general corank-1 modification of any bundle $E$ with $h^0(E)>0$ at a general point reduces $h^0$
by 1. Therefore we can conclude the following which will be useful in the next section:
\begin{cor} Notations as in Theorem \ref{g thm}, we have:
\par Case (A): a general locally corank-1 down modification of $N(1)$ at at most $ \chi(N(1))=2e+(n-3)(1-g)$
points  has
$H^1=0$.\par
Case (B): a general locally corank-1 down modification of $\check N(2)$ at 
at most $ \chi(\check N(2))=e(n-3)+(n+1)(1-g)$
points has $H^1=0$.

\end{cor}
\section{Curves on anticanonical hypersurfaces}\label{d=n}
Here we construct some curves $C$ of large degree
 with the NMR or CMR property on hypersurfaces $X$ of degree $n$ in $\P^n$.
If $C$ has degree $e$ and genus $g$ and normal bundle $N=N_{C/X}$,
 we have  \[\chi(N)=e+(n-4)(1-g)\] which is $\geq 0$ provided $e\geq (n-4)(g-1)$.
Also \[\chi(N(-1)=-(n-3)e+(n-4)(1-g)<0.\]
So what has to be proven is the following..
\begin{thm}\label{d=n thm}
 Let $X$ be a general hypersurface of degree $n$ in $\P^n, n\geq 4$. Then 
For all $g\geq 0$ and $e\geq 12n^2g$, $X$ contains a curve $C$ of genus $g$ and degree $e$
with normal bundle $N$ such that\par
(A) (i) $H^1(N)=0$;\par
(ii) $H^0(N(-1))=0$.\par
(B)
(i) $H^0(\check N)=0$; \par
(ii) $H^1(\check N(1))=0$.
\end{thm}
\begin{proof}  (A)
(i) is a special case of the results of \cite{elliptic}, \S 4, which yield curves $C$ which are balanced,
because balanced clearly implies $H^1(N)=0$.\par
For (ii) we use the singular model used in the proof of Theorem
37 in \cite{elliptic}:
\[C_0=C_1\cup C_2\subset X_0=X_1\cup_F X_2\] 
Thus
$X_2$  is a general hypersurface of degree 
$n-1$ in $\P^n$, $X_1$ is the blowup of $\P^{n-1}$ in a complete intersection
$Y=F_{n-1}\cap F_n$ where $F\subset X_1$ is the birational transform of $F_{n-1}$
and a hyperplane section of $X_2$. Writing
\[e=ne_1-a, a\leq n,\]
let $C_2\subset X_2$ be a disjoint union of $e_1(n-1)-a$ lines
with trivial normal bundle  in $X_2$  $C_1\subset X_1$ corresponds to a curve $C'_1\subset\P^{n-1}$
of genus $g$ and degree $e_1$meeting $Y$ in $a$ general points with general tangents. We use  the following limit of $\O(1)$
\[\L'_1=\O_{X_1}(1-F)\cup\O_{X_2}(1).\]
We have 
\eqspl{O(1)}{\O_{C_1}(1)\simeq C'_1.F_n\setminus C'_1.Y, C_1.F=C'_1.F_{n-1}\setminus C'_1.Y}
Therefore
 we can identify 
 \[\O_{C_1}(\L'_1)=\O_{C'_1}(1)\] while $N_{C_1/X_1}$ is a down modification
of $N_{C'_1/\P^{n-1}}$ at the points of $C'_1\cap Y$. Also,
as each component of $C_2$ is a line with trivial normal bundle, we have  $N_{C_2}(-\L'_1)=N_{C_2}(-1)=(n-2)\O(-1)$. 
Therefore $H^0((N_{C_0/X_0}(-\L'_1)$ corresponds to
sections a subsheaf of $N_{C'_1/\P^{n-1}}(-1)$ vanishing on $F_n.C'_1-Y.C'_{n-1}$ and the latter
divisor certainly contains a divisor of type $C'_1.\O(1)$. Because we may assume
  $H^0(N_{C'_1/\P^{n-1}}(-2))=0$  by Theorem \ref{g thm}, it follows that
$H^0(N_{C_0/X_0}(-1))=0$.\par
(B) We will prove (ii) as (i) is similar and simpler, and in case $g\geq 1$ also follows from A(i). We will use the same
fang and curve as in (A) but with the line bundle $\L_1=\O(1)\cup\O(1-F)$ opposite of the choice for (A). 
Note that with this choice $\L_1$ is trivial on $C_2$ hence
\[H^1(\check N_{C_2/X_2}(\L_1)(-F)=0.\]
Therefore, to prove $H^1(\check N_{C_0/X_0}(\L_1))=0$ it suffices to prove $H^1(\check N_{C_1/X_1}(1))=0$.
By \eqref{O(1)},we can identify $\O_{C_1}(1)=\O_{C'_1}(n-C'_1.Y)$ and this contains $\O_{C'_1}(2)$
as subsheaf and of course $\check N_{C_1/X_1}$ contains $\check N_{C'_1/\P^n}$ as subsheaf. Therefore the vanishing
$H^1(\check N_{C'_1/\P^n}(2))$ as proved in Theorem \ref{g thm} implies $H^1(\check N_{C_1/X_1}(1))=0$
as claimed. 
\end{proof}
\section{Lower-degree Fano hypersurfaces}\label{d<n}
Here we consider curves on hypersurfaces of degree $d<n$ in $\P^n$.
In \cite{elliptic} we constructed some such curves whose normal bundle is balanced, under some
restrictive conditions on the degree and genus of the curve. Here we show under much less restrictive conditions
the existence of curves with NMR or CMR property.
\begin{thm}\label{subcanonical}
Let $X$ be a general hypersurface of degree $d$ in $\P^n$ with $1<d\leq n-1$.
Then\par (A) For all $g\geq 0$ and \[e>(g+1)((d-1)d/2+1), d\geq5\]
or \[e>(g+1)(1+2d), d<5,\] $X$ contains a curve of genus $g$ and degree $e$
with normal bundle $N$ such that\par(i) $H^1(N)=0$;\par
(ii) $H^0(N(-1))=0$.\par
(B) If $d\geq 4$ then, for all $g\geq 0$ and $e\geq \frac{(g+1)d(d-1)}{(d-2)}$,   
$X$ contains a curve of genus $g$ and
degree $e$ and conormal bundle $\check N$ such that\par
(i) $H^1(\check N(1))=0$;\par
(ii) $H^0(\check N)=0$.
\end{thm}
\begin{proof}
(A)(i) For certain $g,e$, Theorem 41 of \cite{elliptic} yields curves such that $N$ is balanced and moreover
$\chi(N)\geq 0$, which implies $H^1(N)=0$. We will use the same construction to show that in the more 
general case considered here one still has $H^1(N)=0$. Thus consider a fang of type $(n, m=d-1)$
\[X_0=X_1\cup_Z X_2\subset P_0=P_1\cup P_2,\]
Thus $X_1\subset P_1$ is the birational transform of a hypersurface of degree $d$ with multiplicity
$d-1$ along $\P^{n-m-1}$, and is itself a $\P^{n-m-1}$-bundle over $\P^m$ of the form $\P(G)$
where $G$ fits in an exact sequence
\[\exseq{\O(-m)}{\O(1)\oplus(n-m)\O}{G}\]
$X_2$ is the birational transform of a hypersurface of degree $d$ containing $\P^m$, and is fibred over
$\P^{n-m-1}$ with fibre a hypersurface of degree $m=d-1$ in the $\P^{m+1}$ fibre of $P_2$ over 
$\P^{n-m-1}$.\par
As in \cite{elliptic} we consider a lci nodal curve
\[C_0=C_1\cup C_2\]
where $C_1\subset X_1$ is a curve of genus $g$, mapping to a degree $e$ curve in $\P^n$
and to a degree $e_0$ curve  $C'_1\subset \P^m$, and $C_2\subset X_2$ is a disjoint union of $(e-e_0)$
lines in fibres which we know have trivial normal bundle, hence
\eqspl{c2}{H^1(N_{C_2/X_2}(-Z)=0.}
Thus, to conclude $H^1(N_{C_0/X_0})=0$ it will suffice to prove that $H^1(N_{C_1/X_1})=0$.
The map $C_1\to X_1$ lifting $C'_1\to\P^m$ corresponds to an exact sequence
\[\exseq{K}{G|_{C'_1}}{M}\]
where $M$ a line bundle of degree $e$ (i.e. $\O_{\P^n}(1)|_{C_1}))$ such that $M-\O_{\P^m}(1)$ is effective
and we have an exact sequence
\eqspl{vertical-horizontal}{\exseq{\check K(M)}{N_{C_1/X_1}}{N_{C'_1/\P^m}}.}
It will suffice to show that $H^1(N_{C'_1/\P^m})=0=H^1(\check K(M))$.\par
Now it is well know that curves (even with general moduli) with vanishing of $H^1(N_{C'/\P^m})$
exist as soon as the Brill-Noether number $\rho=g-(m+1)(g-e_0+m)\geq 0$. So it suffices to show
that for a general $C'_1, M$ one has $H^1(\check K(M)=0$. Note that $\check K(M)$ has slope independent of $g$
\[\mu(\check K(M))=\frac{(n-d+1)e-de_0}{n-d}.\]
\par Now
vanishing of $H^1(\check K(M))$ can be proven by induction on the genus $g$
as in the proof of Theorem 41 in \cite{elliptic}. For $g=0$ we have by \cite{caudatenormal}, Theorem 31
that $K$ is balanced, hence so is $\check K(M)$. We claim that the latter bundle has slope$\geq 1$
which clearly implies $H^1(\check K(M))=0$.
The slope condition read
\[e\geq \frac {n-d}{n-d+1}+\frac{d}{n-d+1}e_0\]
which clearly holds provided $e\geq 1+de_0/2$. Thus choosing $e_0=\min(m,4)$
we get that $\check K(M)|_{C_1}$ for the general $C_1$ of genus 0 and degree
\[e\geq\max(d(d-1)/2, 2d).\]
We assume for simplicity $m\geq 4$, i./e. $d\geq 5$ as the case $m<4$ is similar and simpler.
For the induction step we work as in \cite{elliptic}, proof of Theorem 41, use a 2-nodal curve of the form
\[D'_0=D'_1\cup_{p,q}D'_2\subset\P^m\]
with $D'_1, D'_2$ of genus 0 resp, $g-1$ and degree $e_{01}\geq m$ resp. $e_{02}\geq mg$.
Now the result of \cite{ununs}, Corollary 3 says in our case that $D'_0$ is projectively normal provided
\[e_{01}\geq 4, e_{02}\geq 2g+2.\]
Since these conditions are automatic in our case, it follows as in 
the above-referenced proof that  $D'_0$ can be lifted to $D_0\subset X_1$ corresponding to 
 a line bundle $M=M_1\cup M_2, M_i=M|_{D'_i}$ of degree 
 \[e_1=e_{11}+e_{12}, e_{11}=\deg(M_1)\geq d(d-1)/2, e_{12}=\deg(M_2)\geq 1+gd(d-1)/2\] 
 with $M(-H)$ effective where $H\sim\O_{D'_0}(1)$,
such that $H^1(\check K|_{D_2}(M_2)=0$ (by induction), and $H^1(\check K_{D_1}(M_1)(-p-q)=0$
(by balancednes plus degree). Then $H^1(\check K_{D_0}(M))=0$. This completes the induction step
and the proof of Theorem \ref{subcanonical}(A)(i).\par
The proof of A(ii) is similar and simpler. We use $\L_1=\O(1)\cup\O(1-Z)$ as the limit of $\O(1)$
so that $N_{C_0/X_0}(-\L_1)|_{C_2}(-Z)$ is a direct sum of $\O(-1)$ on each component hence
has $H^0=0$. Hence it suffices to prove $H^0(N_{C_1/X_1}(-H_{\P^n}))=0$ where $H_{\P^n}$
denotes the $\O(1)$ from $\P^n$, i.e. $M$ in the above notation. But because $M$ contains
$\O_{\P^m}(2)$, the vanishing of $H^0$ follows from Theorem \ref{g thm}.\par
For (B), note that
\[\chi(\check N(1))=e(d-3)+n(1-g).\]
We begin by proving $H^1(\check N(1)=0)$ for $d\geq 4$, first for genus 0.
We use the same singular model $C_0=C_1\cup C_2$ as in Case (A), taking again
$\L_1=\O(1)\cup\O(1-Q)$. As in \eqref{c2},
triviality of the normal bundle of $C_2$ yields
\[H^1(\check N(\L_1)|_{C_2}(-Z))=0\]
so it suffices to prove $H^1(\check N(1)_{C_1})=0$. Dualizing the exact sequence
\eqref{vertical-horizontal}, we get exact
\eqspl{horizontal-vertical}{\exseq{\check N_{C'_1/\P^m}(M)}{\check N_{C_1/X_1}(1)}{K}.}
As we have seen, the kernel and cokernel in this sequence are balanced with 
\[\chi(\check N_{C'_1/\P^m}(M))=e(d-2)-(e_0-1)d\]
\[\chi(K)=de_0-e+n-d.\]
As these are balanced bundles, $\chi\geq 0$ or $\mu\geq -1$ implies $H^1=0$.
So it suffices to take
\[e\geq d(d-1)/(d-2), e_0\in[e/d, (d-2)e/d]\]
(note that the interval has length $\frac{(d-3)e}{d}>1$ hence contains an integer).
This will ensure that $H^1(\check N(1))=0$.\par
Again the proof that $H^0(\check N)=0$ is similar and simpler.
This completes the proof for genus 0.\par
For  $g>0$ the argument proceeds by induction as above. Using \eqref{horizontal-vertical}
we can show as above that $H^1(\check N_{C'_1/\P^m}(M))=0$ (note $m\geq 3$), and it suffices to show
$H^1(K)=0$. Again we can use the 2-node degeneration $D'_0=D'_1\cup_{p,q} D'_2$ as above
and note that $H^1(K_{D'_1}(-p-q))=0$ as $K|_{D'_1}$ has slope $\geq 1$ and $H^1(K|_{D'_2})=0$ by induction
so $H^1(K_{D'_0})=0$. Again the proof that $H^0(\check N)=0$ is similar and simpler.

\end{proof}
	\bibliographystyle{amsplain}
	\bibliography{../mybib}
	\end{document}